\documentclass[12pt,a4paper]{article}
\usepackage{amssymb} 
\usepackage{amsmath} 
\usepackage{amsthm} 
\usepackage{hyperref}
\usepackage{todonotes}
\usepackage{verbatim}
\usepackage{a4wide}

\newtheorem{theorem}{Theorem}
\newtheorem{lemma}[theorem]{Lemma}
\newtheorem{corollary}[theorem]{Corollary}
\newtheorem{conjecture}[theorem]{Conjecture}

\newcommand{\Prob}{\,\mathbb{P}}

\newcommand{\const}{c}

\title{Structure and colour in triangle-free graphs}

\author{
N.R. Aravind
\thanks{Department of Computer Science \& Engineering, Indian Institute of Technology Hyderabad, India. Email: \protect\href{mailto:aravind@iith.ac.in}{\protect\nolinkurl{aravind@iith.ac.in}}. Supported by a travel grant from the Science and Engineering Research Board, Department of Science and Technology, Govt of India (project number: MTR/2017/000711). This work was initiated during a one-month visit to Radboud University Nijmegen, the hospitality of which is warmly acknowledged.}
\and
Stijn Cambie
\thanks{Department of Mathematics, Radboud University Nijmegen, Netherlands. 
Email: \protect\href{mailto:stijn.cambie@hotmail.com}{\protect\nolinkurl{stijn.cambie@hotmail.com}}, \protect\href{mailto:r.deverclos@math.ru.nl}{\protect\nolinkurl{r.deverclos@math.ru.nl}}, \protect\href{mailto:ross.kang@gmail.com}{\protect\nolinkurl{ross.kang@gmail.com}}. Supported by a Vidi grant (639.032.614) of the Netherlands Organisation for Scientific Research (NWO).}
\and
Wouter Cames van Batenburg
\thanks{Department of Computer Science, Universit\'e Libre de Bruxelles, Belgium. 
Email: \protect\href{mailto:wcamesva@ulb.ac.be}{\protect\nolinkurl{wcamesva@ulb.ac.be}}. Supported by an ARC grant from the Wallonia-Brussels Federation of Belgium.}
\and
R\'emi de Joannis de Verclos
\footnotemark[2]
\and
Ross J. Kang
\footnotemark[2]
\and
Viresh Patel
\thanks{Korteweg de Vries Institute for Mathematics, University of Amsterdam. Email: \protect\href{mailto:vpatel@uva.nl}{\protect\nolinkurl{vpatel@uva.nl}}. Supported
by the Netherlands Organisation for Scientific Research (NWO) through the Gravitation Programme Networks
(024.002.003).}
}

\begin{document}

\maketitle

\begin{abstract}
Motivated by a recent conjecture of the first author, we prove that every properly coloured triangle-free graph of chromatic number $\chi$ contains a rainbow independent set of size $\lceil\frac12\chi\rceil$.
This is sharp up to a factor $2$.
This result and its short proof have implications for the related notion of chromatic discrepancy.

Drawing inspiration from both structural and extremal graph theory, we conjecture that every triangle-free graph of chromatic number $\chi$ contains an induced cycle of length $\Omega(\chi\log\chi)$ as $\chi\to\infty$.
Even if one only demands an induced path of length $\Omega(\chi\log\chi)$, the conclusion would be sharp up to a constant multiple.
We prove it for regular girth $5$ graphs and for girth $21$ graphs.

As a common strengthening of the induced paths form of this conjecture and of Johansson's theorem (1996), we posit the existence of some $c >0$ such that for every forest $H$ on $D$ vertices, every triangle-free and induced $H$-free graph has chromatic number at most $c D/\log D$. We prove this assertion with `triangle-free' replaced by `regular girth $5$'.

\end{abstract}

\section{Introduction}\label{sec:intro}

For graphs with bounded clique number $\omega$, the tradeoff between chromatic number $\chi$ being high and there being certain induced subgraphs is a central topic in graph theory. This is the context of the famous and longstanding conjecture of, independently, Gy\'arf\'as~\cite{Gya75} and Sumner~\cite{Sum81}, cf.~\cite{ScSe18+}.
This note is solely concerned with this type of problem.

Our starting point is indeed a more explicit form of this tradeoff, where the commodities are instead proper colourings and {\em rainbow} induced subgraphs (that is, ones in which all colours assigned to its vertices are distinct).
This already has some history in the area: for instance, Kierstead and Trotter~\cite{KiTr92} pursued this in attempts towards the Gy\'arf\'as--Sumner Conjecture.
It is interesting in its own right and some recent activity~\cite{BBCF19,GySa16,ScSe17} has been motivated by a conjecture of this form due to the first author.
\begin{conjecture}[Aravind, cf.~\cite{BBCF19}]\label{conj:aravind}
Every properly coloured triangle-free graph of chromatic number $\chi$ contains a rainbow induced path of length $\chi$.
\end{conjecture}

\noindent
By classic results~\cite{Gal68,Has65,Roy67,Vit62,Gya87}, the statement is true when omitting both `triangle-free' and `induced', or omitting `rainbow'. 
It is false when omitting both `triangle-free' and `rainbow', as we discuss below (see Theorem~\ref{thm:offdiagonal}).
Babu, Basavaraju, Chandran and Francis~\cite{BBCF19} proved the statement under the extra assumption that no cycle in $G$ has length less than $\chi$.
Scott and Seymour~\cite{ScSe17} proved a form of it for any fixed clique number $\omega \ge 2$, but with length $f(\chi)$ for some unbounded increasing function $f$ instead of $\chi$.

In the discussion at the end of their note, Scott and Seymour observed that the type of rainbow induced subgraph it makes sense to hope for in this problem setting is already rather simple: it is limited to forests of paths. 
In Section~\ref{sec:proof}, we focus on the simplest possible structure and show the following.

\begin{theorem}\label{thm:main}
For each $r\ge 3$ every properly coloured $K_r$-free graph of chromatic number $\chi$ contains a rainbow independent set of size $\lceil\frac12\chi^{1/(r-2)}\rceil$.
\end{theorem}

\noindent
When $r=3$ this is sharp up to a factor $2$.
The $r=3$ case is a consequence of Conjecture~\ref{conj:aravind} if true, by taking every other vertex in the path.
In Section~\ref{sec:discrepancy}, we discuss this result's implications for the related concept of chromatic discrepancy~\cite{AKSS15}.
There appears to be room for improvement in Theorem~\ref{thm:main} for $r>3$, but much of it may come from the gap in current bounds on off-diagonal Ramsey numbers, as we next discuss.

Consider the smallest independence number $\alpha$ in $K_r$-free graphs as a function of the chromatic number $\chi$.
The following statement follows from the best-to-date asymptotic results on off-diagonal Ramsey numbers. For completeness, its brief derivation is included in the appendix.

\begin{theorem}[\cite{AKS80,BoKe10}]\label{thm:offdiagonal}
For each $r\ge 3$ there are $\const_1,\const_2>0$ such that the following hold.
There is a $K_r$-free graph of chromatic number $\chi$ that contains no independent set of size
$\const_1\chi^{2/(r-1)} (\log \chi)^{(r^2-r-4)/((r-2)(r-1))}$.
Every $K_r$-free graph of chromatic number $\chi$ contains an independent set of size $\const_2\chi^{1/(r-2)}\log\chi$.
\end{theorem}

\noindent
This immediately yields the following result complementing Theorem~\ref{thm:main}.
\begin{corollary}\label{cor:offdiagonal}
For each $r\ge 4$ there are $\const>0$ and a $K_r$-free graph of chromatic number $\chi$ such that no matter the proper colouring it contains no rainbow independent set of size $\const\chi^{2/(r-1)} (\log \chi)^{(r^2-r-4)/((r-2)(r-1))}$.
\end{corollary}

\noindent
For $r\ge4$, it remains possible that the bound in Theorem~\ref{thm:main} could be increased by a logarithmic factor, so that it would match and indeed qualitatively strengthen Theorem~\ref{thm:offdiagonal}. 
On the other hand, intuitively (based on the sharpness of Theorem~\ref{thm:offdiagonal} for $r=3$), improvement by more than a logarithmic factor may be for want of a breakthrough in Quantitative Ramsey Theory. 

Theorems~\ref{thm:main} and~\ref{thm:offdiagonal} hint at the following generalisation of Conjecture~\ref{conj:aravind}.

\begin{conjecture}\label{conj:aravindclique}
For each $r\ge 3$ every properly coloured $K_r$-free graph of chromatic number $\chi$ contains a rainbow induced path of length $\chi^{1/(r-2)}$.
\end{conjecture}

\noindent
The aforementioned result of Scott and Seymour~\cite[1.3]{ScSe17} already constitutes partial progress. 
The statement is true when omitting `rainbow'. It is surprisingly difficult to bound the maximum rainbow induced path length significantly below both the maximum induced path length and the chromatic number.

\medskip

The proof of Theorem~\ref{thm:main} has affinity to a proof of Gy\'arf\'as guaranteeing induced paths of length $\chi$ in triangle-free graphs of chromatic number $\chi$~\cite{Gya87}. 
Returning to the roots concerning `non-rainbow' structure, we have formulated,
based on Theorem~\ref{thm:offdiagonal} and some intuition from the random graph~\cite{Luc93},
the following successively stronger conjectures.
\begin{conjecture}\label{conj:nonrainbow}
There is some $\const>0$ such that every triangle-free graph of chromatic number $\chi$ contains an induced path of length at least $\const\chi\log\chi$.
\end{conjecture}
\begin{conjecture}\label{conj:nonrainbowcycle}
There is some $\const>0$ such that every triangle-free graph of chromatic number $\chi >2$ contains an induced cycle of length at least $\const\chi\log\chi$.
\end{conjecture}
\noindent
By Theorem~\ref{thm:offdiagonal}, each statement if true is sharp up to the respective choices of $\const$. 
Either but instead with `induced star/tree of size $\const\chi\log\chi$' is true due to Johansson's result on the chromatic number of triangle-free graphs~\cite{Joh96}.  
Conjecture~\ref{conj:nonrainbowcycle} is a quantitative strengthening of a conjecture of Gy\'arf\'as~\cite{Gya87}, and slightly stronger than~\cite[Conj.~7]{HoMc02}. Gy\'arf\'as's conjecture was recently confirmed by Chudnovsky, Scott and Seymour~\cite{CSS17}, but the induced cycle lengths guaranteed in~\cite{CSS17} are very small compared to $\chi\log\chi$; see also~\cite{ScSe18}. 

Perhaps Conjecture~\ref{conj:nonrainbowcycle} is difficult in general, but on the other hand, we have managed to obtain some concrete progress under the additional exclusion of one or more cycle lengths. 
\begin{theorem}\label{thm:highergirth}
\begin{itemize}
\item
There is some $\const>0$ such that every regular $C_4$-free graph of chromatic number $\chi >2$ contains an induced cycle of length at least $\const\chi\log\chi$.
\item
For each $g\ge 5$ every girth $g$ graph of chromatic number $\chi>2$ contains an induced cycle of length at least $3+(\chi-1)(\chi-2)^{\lfloor(g-5)/16\rfloor}$.
\end{itemize}
\end{theorem}
\noindent
This in particular implies that Conjecture~\ref{conj:nonrainbowcycle} holds for regular girth $5$ graphs and for girth $21$ graphs. We have not made much effort to optimise the constant $21$, for the method we use seems unlikely to reduce it below $13$ or so. 
   Theorem~\ref{thm:highergirth} also asserts that each girth $5$ graph has an induced cycle of length at least $\chi+2$, which is is not far from best possible since the $C_4$-free process~\cite{BoKe10} yields $n$-vertex $\chi$-chromatic $C_4$-free graphs with independence number $O((n \log n)^{2/3})=O((\chi \log \chi)^2)$.

Johansson's~\cite{Joh96} and Conjecture~\ref{conj:nonrainbow} together naturally prompt another possibility.

\begin{conjecture}~\label{conj:generalisedJohansson}
There is some $\const>0$ such that for every forest $H$, every triangle-free graph containing no induced $H$ has chromatic number at most $\const |V(H)|/\log|V(H)|$.
\end{conjecture}

\noindent  
 If true, Conjecture~\ref{conj:generalisedJohansson} would constitute a common generalisation of Johansson's theorem, Conjecture~\ref{conj:nonrainbow} and the fact that $\chi(G)=O(\alpha(G)/\log\alpha(G))$ for every triangle-free graph $G$, corresponding to the cases where $H$ is a star, a path and an independent set, respectively. 
 The conclusion of Conjecture~\ref{conj:generalisedJohansson} would fail if some $H$ were allowed to contain a cycle, since for each $\ell\geq 3$ there are $C_\ell$-free graphs of arbitrarily large chromatic number.
 Note also that to prove Conjecture~\ref{conj:generalisedJohansson} it suffices (by adding a single vertex connected to all components if needed) to prove it for all trees $H$.
 As a first step, we have proved a form of Conjecture~\ref{conj:generalisedJohansson} for regular girth $5$ graphs (see Corollary~\ref{cor:generalisedJohanssonForRegularGirth5}). 

%

\section{Large rainbow independent sets}\label{sec:proof}

\begin{proof}[Proof of Theorem~\ref{thm:main}]
We carry out an induction on $r\ge 3$.
Let $G$ be a $K_r$-free graph of chromatic number $\chi$ and let $c: V(G)\to {\mathbb Z}^+$ be a proper colouring. We seek a rainbow independent set of size $\lceil\frac12\chi^{1/(r-2)}\rceil$.
Initialise $G'=G$ and $X=\emptyset$, and iterate the following until $G'$ is empty (if needed).
\begin{enumerate}
\item Take an arbitrary vertex $v\in V(G')$ and add it to $X$.
\item Let $S=c^{-1}(c(v))$ and delete the vertices of $S$ from $G'$.
\item Let $N=N_{G'}(v)$ and consider the subgraph $G'[N]$ of $G'$ induced by $N$.
\begin{enumerate}
\item\label{stop} If $\chi(G'[N]) > \chi^{(r-3)/(r-2)}$, then stop the procedure by outputting the largest rainbow independent set in $G'[N]$.
\item Otherwise, delete the vertices of $N$ from $G'$.
\end{enumerate}
\end{enumerate}
Note that if $r=3$, then the condition in~\ref{stop} is vacuous, in which case we are directly proving the base case.
If on the other hand the procedure stops in~\ref{stop} (and so $r\ge 4$), then since $G'[N]$ is $K_{r-1}$-free it contains a rainbow independent set of size $\lceil \frac12\chi(G'[N])^{1/(r-3)} \rceil \ge \lceil\frac12\chi^{1/(r-2)}\rceil$ by induction, in which case we are done.

If the procedure continues until $G'$ is empty, then by construction the final set $X$ is a rainbow independent set, and so it suffices to show that $|X| \ge \frac12\chi^{1/(r-2)}$.
To this end, let $S_i$ and $N_i$ be the vertex subsets $S$ and $N$ respectively in iteration $i\in \{1,\dots,|X|\}$.
Since $\chi(G'[N_i]) \le \chi^{(r-3)/(r-2)}$ for every $i$, $V(G) = \cup_i (S_i\cup N_i)$ certifies a proper colouring of $G$ with at most $|X|(1+\chi^{(r-3)/(r-2)})$ colours. Thus $|X|(1+\chi^{(r-3)/(r-2)}) \ge \chi$ and so $|X|\ge \chi/(1+\chi^{(r-3)/(r-2)}) \ge \frac12\chi^{1/(r-2)}$, as promised.
\end{proof}

We remark that the same argument, i.e.~performing the algorithm above applied to the binomial random graph $G_{n,p}$ (together with standard facts about the model), yields the following result. It is close to best possible: in the first regime it is sharp up to a constant factor, in the second up to a $\log n$ factor.
Recall that a property in $G_{n,p}$ is said to hold {\em asymptotically almost surely (a.a.s.)} if it holds with probability tending to one as $n\to\infty$.

\begin{theorem}\label{thm:random}
	Let $p=p(n)$ satisfy $np=\omega(1)$ and $p=o(1)$.
	\begin{itemize}
		\item 
		Given $1/2 < c\le 1$, suppose $p =o(n^{-c})$ as $n\to\infty$. 
		Then a.a.s.~for any proper colouring of $G_{n,p}$, there is a rainbow independent set of size $\Omega(\chi(G_{n,p}))$.
		\item Suppose $p = \omega(\sqrt{ (\log n)/n})$ as $n\to\infty$.
		Then a.a.s.~for any proper colouring of $G_{n,p}$, there is a rainbow independent set of size 
		$\Omega(1/p)$. 
	\end{itemize}
\end{theorem}

\begin{proof}
	In the first regime, let $C=5/(2c-1).$ 
	We first prove the observation that for every vertex $v$ in $G_{n,p}$, the probability that its neighbourhood induces a graph with maximum degree at least $C$ is at most $2n^{-5}$ as $n\to\infty$.
	Note that this probability increases as $p$ increases, so it is sufficient to prove the statement when $p=n^{-c}/2$ where $1/2<c<1$.
	By the Chernoff bound we know that for $n$ sufficiently large $$\Prob(\deg(v)> 2np) < \exp(-np/3)=\exp(-n^{1-c}/6) <n^{-5}.$$
	The probability that a neighbour $u \in N(v)$ has degree at least $C$ is bounded by 
	$$\binom{ \deg(v) -1 }C p^C \le  (p \deg(v) )^C  .$$
	So if $ \deg(v) \le 2np=n^{1-c},$ then this is bounded by $n^{(1-2c)C}=n^{-5}.$
	
	This observation implies that with probability at least $1-2n^{-4}$ we have $\chi(G[N(v)])<C$ for every vertex $v$ and hence the algorithm gives a rainbow independent set $X$ of size at least $\chi(G_{n,p})/(1+C).$

	In the second regime, note that by the Chernoff bound a.a.s.~every degree of the graph is bounded by $2np.$ Also a.a.s.~the independence number is at most $2p^{-1}\log np<np$.
	Hence a.a.s.~we have 
	$\lvert S_i\rvert + \lvert N_i \rvert  < 3np$ for every $i$ in the algorithm and hence $\lvert X \rvert \ge 1/(3p).$
\end{proof}

\section{Chromatic discrepancy}\label{sec:discrepancy}

In related work, the first author together with Kalyanasundaram, Sandeep and Sivadasan~\cite{AKSS15} studied the notion of {\em chromatic discrepancy}, the least over all proper colourings of the greatest difference between size and induced chromatic number taken over all rainbow subgraphs.
Starting with some triangle-free graph of chromatic number $\chi$ and iterating Theorem~\ref{thm:main}, each time extracting from what remains a large rainbow independent set and all associated colour classes, one finds an induced subgraph $H$ that is rainbow with at least $\chi$ colours and such that $\chi(H) \le \log_2\chi+1$.
\begin{theorem}\label{thm:trianglefreediscrepancy}
Every properly coloured triangle-free graph of chromatic number $\chi$ contains a rainbow induced subgraph on $\chi$ vertices of chromatic number at most $\log_2\chi+1$. \qed
\end{theorem}
\noindent
In other words, the chromatic discrepancy of any triangle-free graph of chromatic number $\chi$ is at least $\chi-\log_2\chi-1$.
 It is an open question whether the logarithmic term can be reduced to some constant independent of $\chi$.
 It was conjectured~\cite[Qu.~4]{AKSS15} that $\chi-\omega$ is a lower bound on the chromatic discrepancy for any graph of chromatic number $\chi$ and clique number $\omega$.  Corollary~\ref{cor:offdiagonal} refutes this for every fixed $\omega\ge 3$; however, iterating Theorem~\ref{thm:main} in the same way as above yields $(1-o(1))\chi$ chromatic discrepancy as $\chi\to\infty$.
 
Similarly iterating Theorem~\ref{thm:random} yields the following for chromatic discrepancy of $G_{n,p}$, an improvement upon~\cite[Thm.~4.6]{AKSS15}. 

\begin{theorem}\label{thm:randomdiscrepancy}
	Let $p=p(n)$ satisfy $np=\omega(1)$ and $p=o(1)$.
	\begin{itemize}
		\item Given $1/2 < c<1$, suppose $p =o(n^{-c})$ as $n\to\infty$.
		Then a.a.s.~for any proper colouring of $G_{n,p}$, there is a rainbow induced subgraph on $\chi(G_{n,p})$ vertices of chromatic number at most $O(\log \chi(G_{n,p}))$.
		\item 
		Given $0\le c<1$, suppose $p = \omega(n^{-c})$ as $n\to\infty$.
		Then a.a.s.~for any proper colouring of $G_{n,p}$, there is a rainbow induced subgraph on $\chi(G_{n,p})$ vertices of chromatic number at most
		\[O(-\log p \cdot \max{\left\{p, \sqrt{ (\log n)/(np)}\right\} }   \chi(G_{n,p})).\] 
	\end{itemize}
\end{theorem}

\begin{proof}
	We will essentially iterate the algorithm in Theorem~\ref{thm:main}.
	
	Initialise $G''=G_{n,p}$ and $Y=\emptyset$, and iterate the following until $Y$ contains at least $\chi(G_{n,p})$ vertices.
	
\begin{enumerate}
	\item Initialise $G'=G''$ and $X=\emptyset$, and iterate the following until $G'$ is empty.
		\begin{enumerate}
			\item Take an arbitrary vertex $v\in V(G')$ and add it to $X$.
			\item Let $S=c^{-1}(c(v))$ and delete the vertices of $S$ from $G'$.
			\item Delete the vertices of $N_{G'}(v)$ from $G'$.
		\end{enumerate}
	\item Delete the vertices in $c^{-1}(c(X))$ from $G''.$
	\item Add the vertices from $X$ to $Y$.
\end{enumerate}
	
	In the first regime, we have seen in the proof of Theorem~\ref{thm:random} that a.a.s.~in every step in the algorithm the chromatic number of $N_i$ is at most $C=5/(2c-1)$. So in every iteration, we have selected at least $\frac 1{C+1} \chi(G'')$ vertices.
	This implies that we need to perform at most $\log \chi(G_{n,p})/\log(1+\frac 1C)$ iterations to create a rainbow induced subgraph on $\chi(G_{n,p})$ vertices.
	
	In the second regime, a.a.s.~every vertex has degree at most $2np$ and the chromatic number of every neighbourhood is bounded by
	\[f(n,p):=\mathbb{E}(\chi(G_{2np,p}))+\sqrt{8np\log n},\]
due to a result of Shamir and Spencer~\cite{SS87}. Also $\chi(G_{n,p})\sim\frac{np}{2\log{np}}$ a.a.s.
	So in every iteration of the algorithm, we have selected at least $\frac 1{f(n,p)+1} \chi(G'')$ vertices.
	So it takes $$O\left( -\log p \cdot (f(n,p)+1) \right)$$ iterations to select at least $\chi(G_{n,p}) - p \chi(G_{n,p}) $ vertices.
	If $\mathbb{E}(\chi(G_{2np,p}))>\sqrt{8np\log n}$, we have $\mathbb{E}(\chi(G_{2np,p})) \sim \frac{ 2np^2}{2 \log{np^2}} = O(p \chi(G_{n,p}) ).$
	In the other case, we have $f(n,p)=O( \sqrt{ (\log n)/(np)} \chi(G_{n,p})).$
	
	After that, at most $p \chi(G_{n,p}) $ additional distinctly-coloured vertices are needed to form a rainbow induced subgraph on $\chi(G_{n,p})$ vertices, the resulting graph having chromatic number $O( -\log p \cdot p \chi(G_{n,p}) )$.
\end{proof}

\section{Long induced paths and cycles}
\label{sec:highergirth}

This section is devoted to establishing Theorem~\ref{thm:highergirth} and related results.

\begin{lemma}\label{le:boundedcodegree}
For each $t,d \ge 2$, every $K_{2,t}$-free graph of minimum degree $d$ contains induced cycles of $\lceil \frac{d-1}{t-1} \rceil$ distinct lengths, and, in particular, some induced cycle of length at least $2+\lceil \frac{d-1}{t-1} \rceil$. 
\end{lemma}
\begin{proof}
Let $P=p_1,p_2,\ldots $ be an induced path in $G$ of maximal length. Its first vertex $p_1$ has at least $d-1$ neighbours in $V(G)\backslash V(P)$; let us call them the \emph{pending vertices}. By maximality of $P$, each pending vertex has at least one neighbour in $V(P)\backslash \left\{p_1\right\}$. For a pending vertex $v$, we say  $p_j\in V(P)\backslash\left\{p_1\right\}$ is the \emph{first neighbour of $v$} if $vp_j \in E(G)$ and $vp_i \notin E(G)$  for every $2\le i < j$. Note that in that case $v,p_1,\ldots,v_{j-1},v_j$ is an induced cycle of length $j+1$.
At most $t-1$ pending vertices can have a common first neighbour in $P\backslash \left\{p_1\right\}$, since otherwise $K_{2,t}$ would be a subgraph of $G$. It follows that at least $\frac{d-1}{t-1}$ distinct vertices in $V(P)\backslash \left\{p_1\right\}$ are the first neighbour of some pending vertex. Thus $G$ has induced cycles of at least $\frac{d-1}{t-1}$ distinct lengths.
\end{proof}

It turns out there exist induced cycles of length exponential in the girth.

\begin{theorem}\label{thm:largegirth}
For each $k\ge 0$, 
every graph of girth at least $16k+5$ and minimum degree $d\geq 2$ contains an induced cycle of length at least 
\(
3+d(d-1)^{k}.
\)
In particular, if $k\ge 1$, it contains an induced cycle of length $\Omega(d^2)$. 
\end{theorem}

\noindent
{\em Nota bene}: the first part of this proof closely follows that of~\cite[Prop.~6]{KuOs03}.

\begin{proof}
Let $G$ be a graph with minimum degree $d$ and girth $g$.
For a nonnegative integer $r$ and a vertex $v$ in $G$, we let $B_r(v):= \left\{x \in V(G) \mid d_G(x,v) \le r \right\} $ denote the ball of radius $r$ centred at $x$. Let $X$ be a maximal set of vertices that are pairwise at distance at least $2k+1$. Observe that the balls of radius $k$ centred at the vertices of $X$ are pairwise disjoint. Moreover, each vertex is at distance at most $2k$ from $X$.
We extend the collection of balls $(B_k(x))_{x\in X}$ to a partition of $V(G)$ as follows. First add each vertex at distance $k+1$ from $X$ to one of the balls to which it is adjacent. Then add each vertex at distance $k+2$ from $X$ to one of the parts constructed in the previous step. Continue in this way until all vertices of $G$ are covered. For each $x\in X$, denote by $T(x)$ the graph induced by the part obtained from $B_k(x)$ in this way. Because $G$ has girth at least $4k+2$, each $T(x)$ is an induced subtree of $G$. Each non-leaf of the subtree induced by $B_k(x)$ has degree at least $d$, so $T(x)$ has at least $d(d-1)^{k-1}$ leaves, and thus $T(x)$ sends at least $d(d-1)^{k}$ edges to other trees. Moreover, the fact that $g\ge 1+ 2\cdot(4k+1)$  implies that $T(x_1)$ and $T(x_2)$ are joined by at most one edge, for any two distinct $x_1,x_2\in X$. Therefore the minor $G'$ obtained by contracting the trees has minimum degree at least $d(d-1)^{k}$. 
Since $g\ge 1+ 4\cdot(4k+1)$, $G'$ must have girth at least $5$. This allows us to apply Lemma~\ref{le:boundedcodegree} (with $t=2$), together with the girth $5$ condition, yielding an induced cycle of length at least $3+d(d-1)^{k}$ in $G'$. Note that for any two vertices $x,y\in V(G')$, $x$ and $y$ are adjacent if and only if their pre-images in $G$ are joined by precisely one edge. We conclude that $G$ has an induced cycle of length at least $ 3+d(d-1)^{k}.$
\end{proof}

We remark that Theorem~\ref{thm:largegirth} for $k=0$ is sharp when $d=2$ and $d=3$, by $C_5$ and the Petersen graph, respectively. On the other hand, it is conceivable for $k=0$ that one could guarantee an induced cycle of length $\Omega(d^{3/2})$ as $d\to\infty$, which would be best possible for infinitely many values of $d$, due to the Erd\H{o}s-Renyi orthogonal polarity graph, cf.~e.g.~\cite{MuWi07}.

Every graph with chromatic number $\chi$ has an induced subgraph with minimum degree at least $\chi-1$. 
The second part of Theorem~\ref{thm:highergirth} thus follows immediately from Theorem~\ref{thm:largegirth}.


The following corollary (with $t=2$) implies the first part of Theorem~\ref{thm:highergirth}.
\begin{corollary}
For each $t\ge 2$ there is some $\const>0$ such that every regular $K_{2,t}$-free non-forest graph of chromatic number $\chi$ contains an induced cycle of length at least $\const\chi \log \chi$.
\end{corollary}
\begin{proof}
Given a $K_{2,t}$-free graph with maximum degree $\Delta$ and an arbitrary vertex $v$, each neighbour of $v$ has at most $t-1$ neighbours in $N(v)$. Therefore the number of edges in the induced subgraph on the set of all neighbours of any vertex does not exceed $\frac{t-1}{2} \Delta$. This together with the result of Alon, Krivelevich and Sudakov~\cite{AKS99} implies that every $K_{2,t}$-free graph with maximum degree $\Delta$ has chromatic number $\chi= O( \Delta/\log \Delta)$ as $\Delta\to\infty$, and hence $\Delta = \Omega(\chi \log \chi)$ as $\chi\to\infty$. Now combine this with the consequence of Lemma~\ref{le:boundedcodegree} that every $\Delta$-regular $K_{2,t}$-free non-forest graph has an induced cycle of length $\Omega(\Delta)$.
\end{proof}

Lemma~\ref{le:boundedcodegree} in particular shows that girth $5$ graphs contain induced paths of linear length. 
In fact they contain many such paths. We hope that this might become useful towards further progress in Conjectures~\ref{conj:aravind} and~\ref{conj:nonrainbow}.

\begin{lemma}\label{le:girth5inducedpaths}
In any graph of girth at least $5$ and minimum degree $d \ge 2$, there are $d!$ distinct induced paths of order $d+2$ starting at any vertex.
\end{lemma}
\begin{proof}
We apply induction on $d\ge 2$.
Let $G$ be a graph of girth at least $5$ and minimum degree $d$ and let $v\in V(G)$.
 If $d=2$, then there must be a cycle of $G$ containing $v$. We may assume that this cycle is an induced cycle of length at least $5$, and therefore $v$ is an endvertex of two distinct induced paths of order at least $4=d+2$. So we may assume that $d\ge 3$. For any neighbour $w$ of $v$, let $G_w$ denote the connected component containing $w$ in the graph obtained by deleting $v$ and $N(v) \backslash \left\{w\right\}$. 
No vertex of $G_w$ can have more than one neighbour in $\left\{v\right\} \cup N(v)\backslash \left\{w \right\}$, for otherwise $G$ would contain a triangle or a $4$-cycle. It follows that the minimum degree of $G_w$ is at least $d-1$. Hence induction yields that for each $w \in N(v)$, there are at least $(d-1)!$ induced paths in $G_w$ of order $d+1$, starting in $w$. By appending $v$ to these paths, we obtain $(d-1)!$ distinct induced paths of order $d+2$ that start in $v$. Since there are at least $d$ choices for $w$, the lemma follows.
\end{proof}

Lemma~\ref{le:girth5inducedpaths} on induced paths can be extended to rooted induced forests as follows. Roughly speaking, the following says that in any girth $5$ graph with large minimum degree, every large forest occurs many times as an induced subgraph.

\begin{lemma}\label{le:rootedforests}
Let $G$ be a graph of girth at least $5$ and minimum degree $d$. Let $T$ be a forest on $d$ vertices, with $t$ components $T_1, \ldots, T_t$. For each $1\le i \le t$, let $u_i$ be any vertex of $T_i$. Let $S:=\left\{v_1,\ldots, v_t\right\}$ be any size $t$ independent set of $G$. Then there exists an injective graph homomorphism $f: V(T) \to V(G)$ such that
\begin{itemize}
\item $f(u_i)=v_i$ for all $1\le i \le t$, and
\item $f(V(T))$ induces a copy of $T$ in $G$. 
\end{itemize}
\end{lemma}
\begin{proof}
We apply induction on $n:=|V(G)|$. There is nothing to prove for $n=1$, so suppose $n>1$ and assume the result is true for all graphs on fewer than $n$ vertices. 
Let $u(1),\ldots, u(k)$ denote the neighbours of $u_1$ in $T_1$, and let $T'$ be the forest obtained from $T$ by deleting $u_1$. Furthermore, denote by $T(1),\ldots T(k)$ the components of the subforest $T_1\backslash \left\{u_1\right\}$.
Because $G$ has no triangles or $4$-cycles, any two vertices in $N(v_1)$ have no common neighbour other than $v_1$. Therefore $v_1$ has at least $|N(v_1)|- (t-1)\ge d -(t-1)  \ge |V(T_1)| >  k$ neighbours that are not incident to any vertex of $S$ other than $v_1$. Thus there exists a set $N':=\left\{v(1),\ldots, v(k)\right\}$ of $k$ distinct neighbours of $v_1$, such that $S'=S\cup N' \backslash \left\{v_1\right\}$ is an independent set of $G$. Let $G'$ denote the graph obtained from $G$ by deleting $v_1$ and $N(v_1)\backslash N'$. 
Because $G$ has girth at least five, the minimum degree of $G'$ is at least $d-1$. Moreover, $T'$ is a forest on $d-1$ vertices, with components $T(1),\ldots, T(k), T_2,\ldots, T_t$. Recall furthermore that $S'$ is an independent set of $G$, and hence of $G'$.
Thus, by induction, we know that there is a mapping $f':V(T')\to V(G')$  such that $f'(u_i)=v_i$ for all $2\le i\le t$, $f(u(j))=v(j)$ for all $1\le j\le k$, and $f'(V(T'))$ induces a copy of $T'$ in $G'$.
Now we can extend $f'$ to the desired mapping $f$ by defining $f(x)=f'(x)$ for all $x \in V(T')$, and $f(u_1)=v_1$.
\end{proof}

\begin{corollary}\label{cor:generalisedJohanssonForRegularGirth5}
There is some $\const>0$ such that for every forest $H$, every regular girth $5$ graph containing no induced $H$ has chromatic number at most $\const |V(H)|/\log|V(H)|$.
\end{corollary}
\begin{proof}
Let $G$ be a $\Delta$-regular girth $5$ graph. 
Let $H$ be a forest that does not occur as an induced subgraph of $G$. Then by Lemma~\ref{le:rootedforests}, $H$ must have more than $\Delta$ vertices. Combining this with Johansson's theorem~\cite{Joh96} yields $\chi(G) \leq \const'\Delta/\log\Delta \leq \const'|V(H)|/\log|V(H)|$, for some $\const'>0$ and all sufficiently large $|V(H)|$. From this the corollary easily follows.
\end{proof}

\subsection*{Acknowledgement}

We thank Fran\c{c}ois Pirot for an observation using Johansson's Theorem. We are grateful to Gwena\"el Joret and Piotr Micek for helpful discussions in relation to Section~\ref{sec:highergirth}.
We are also grateful to Gwena\"el Joret for his help in identifying a subtlety in an earlier version of this work.

\bibliographystyle{abbrv}
\bibliography{colind}

\appendix

\section{Large independent sets}

\begin{proof}[Proof of Theorem~\ref{thm:offdiagonal}]
Next let $G_n$ be a sequence of instantiations of the final output of the $K_r$-free process on $n$ vertices such that $\alpha(G_n) = O(n^{2/(r+1)}(\log n)^{1-1/(\binom{r}{2}-1)})$ as $n\to\infty$~\cite{BoKe10}.
Thus $\chi(G_n) \ge n/\alpha(G_n) =\Omega(n^{1-2/(r+1)}/(\log n)^{1-1/(\binom{r}{2}-1)})$, from which it follows that $n=O(\chi(G_n)^{(r+1)/(r-1)}(\log \chi(G_n))^{(r^2-r-4)/((r-2)(r-1))})$ as $\chi(G_n) \to\infty$.
From this $\alpha(G_n) = O(\chi(G_n)^{2/(r-1)} (\log \chi(G_n))^{(r^2-r-4)/((r-2)(r-1))})$.

Let $G$ be a $K_r$-free graph of chromatic number $\chi$ with $n$ vertices.
By a classic result of Ajtai, Koml\'os and Szemer\'edi~\cite{AKS80}, the independence number $\alpha$ of $G$ satisfies $\alpha = \Omega(n^{1/(r-1)}(\log n)^{1-1/(r-1)})$ as $n\to\infty$. 
Moreover, from a sequence of iterations of this result (cf.~e.g.~\cite[pp.124--5]{JeTo95}) it follows that the chromatic number of $G$ satisfies $\chi = O((\frac{n}{\log n})^{1-1/(r-1)})$, implying $n = \Omega(\chi^{(r-1)/(r-2)}\log \chi)$ as $\chi\to\infty$. 
From this it follows that $\alpha = \Omega(\chi^{1/(r-2)}\log\chi)$.
\end{proof}

\end{document}